\documentclass[11pt]{amsart}
\usepackage{enumerate}

\newtheorem{theorem}{Theorem}[section]
\newtheorem{lemma}[theorem]{Lemma}
\newtheorem*{lemma*}{Generalized Klingenberg's Lemma}
\newtheorem*{claim*}{Claim}
\newtheorem{corollary}[theorem]{Corollary}
\newtheorem{proposition}[theorem]{Proposition}

\theoremstyle{definition}

\newtheorem{example}[theorem]{Example}

\theoremstyle{remark}
\newtheorem{remark}[theorem]{Remark}
\allowdisplaybreaks
\numberwithin{equation}{section}

\begin{document}
	\title[Convexity Radius and Injectivity Radius decay]{Local Estimate on Convexity Radius and Decay of Injectivity Radius in a Riemannian manifold}
	
	
	\author{Shicheng Xu}
	\address{School of Mathematical Sciences, Capital Normal University, Beijing, 100048, P.R.C.}
	\email{shichengxu@gmail.com}
	
	
	
	
	\date{April 9, 2017}
	
	
	\begin{abstract}
		In this paper we prove the following pointwise and curvature-free estimates on convexity radius, injectivity radius and local behavior of geodesics in a complete Riemannian manifold $M$:
		\begin{enumerate}
			\item the convexity radius of $p$, $\operatorname{conv}(p)\ge \min\{\frac{1}{2}\operatorname{inj}(p),\operatorname{foc}(B_{\operatorname{inj}(p)}(p))\}$, where $\operatorname{inj}(p)$ is the injectivity radius of $p$ and $\operatorname{foc}(B_r(p))$ is the focal radius of open ball centered at $p$ with radius $r$;
			\item for any two points $p,q$ in $M$,
			$\operatorname{inj}(q)\ge \min\{\operatorname{inj}(p), \operatorname{conj}(q)\}-d(p,q),$ where $\operatorname{conj}(q)$ is the conjugate radius of $q$;
			\item for any  $0<r<\min\{\operatorname{inj}(p),\frac{1}{2}\operatorname{conj}(B_{\operatorname{inj}(p)}(p))\}$, any (not necessarily minimizing) geodesic in $B_r(p)$ has length $\le 2r$.
		\end{enumerate}
		We also clarify two different concepts on convexity radius and give examples to illustrate that the one more frequently used in literature is not continuous.
	\end{abstract}
	\maketitle
	\section{Introduction}
	Let $M^n$ be a complete Riemannian manifold of dimension $n$ without boundary. The \emph{injectivity radius} of a point $p\in M$, $\operatorname{inj}(p)$, is defined to be the supremum of radius of open metric balls centered at $p$ that contains no cut points of $p$. The \emph{conjugate radius} of $p$, $\operatorname{conj}(p)$, is defined as the supremum of the radius of open balls centered at the origin of the tangent space $T_pM$ which contains no critical point of exponential map $\exp_p:T_pM\to M$.
	An open set $U\subset M$ is called \emph{convex}, if any two points in $U$ are joined by a unique minimal geodesic in $M$ and its image lies in $U$.
	Following Klingenberg (\cite{Kl}), we call an open ball $B_r(p)$ \emph{strongly convex}\footnote{The terminology here is different from that in \cite{Alexander1978Local, IsaacChavel2006Riemannian}, where the ``strongly convex" coincides with the ``convex" used in this paper for open sets.}, if any $B_s(q)\subset B_r(p)$ is convex.
	The \emph{convexity radius} of $p$ (\cite{Kl}), $\operatorname{conv}(p)$, is defined  to be the supremum of radius of strongly convex open balls centered at $p$. 
	The injectivity radius of $M$ is defined by
	$\operatorname{inj}(M)=\inf_{p\in M}\operatorname{inj}(p)$.  The conjugate radius of $M$, $\operatorname{conj}(M)$ and the convexity radius of $M$, $\operatorname{conv}(M)$ are defined similarly as $\operatorname{inj}(M)$.
		
	It is a classical result by Whitehead (\cite{Wh}) that $\operatorname{conv}(p)>0$. Let $C$ be a compact set containing  $B_r(p)$, and let $K$ be the supremum of sectional curvature on $C$, then the convexity radius of $p$ satisfies the following lower bound (cf. \cite{Wh,ChEb,Kl}),
	\begin{equation}\label{ineq-conv-classical-curv}
	\operatorname{conv}(p)\ge \min\{r, \frac{\pi}{2\sqrt{K}}, \frac{1}{2}\operatorname{inj}(C)\},
	\end{equation} 
	where $\frac{\pi}{\sqrt{K}}$ is viewed as $\infty$ for $K\le 0$.
	
	Our first main result in this paper is a new estimate on convexity radius, which is local, curvature-free and improves Whitehead's theorem. Let the \emph{focal radius} of $p$ (\cite{Dib}), denoted by $\operatorname{foc}(p)$, be defined as the supremum of the radius of open balls centered at the origin of $T_pM$ such that for each tangent vector $v$ in the ball and any normal Jacobi field $J$ along the radial geodesic $\exp_ptv$ with $J(0)=0$, $\frac{d}{dt}|J(t)|> 0$ $(0<t\le 1)$.  It is the largest radius of balls on the tangent space of $p$ such that the distance function to the origin remains strictly convex with respect to the pullback metric (see \cite{Dib} or Lemma 2.4).
\begin{theorem}\label{thm-conv-lowerbound}
	Let $U=B_{\operatorname{inj}(p)}(p)$ be the open ball centered at $p$ of radius $\operatorname{inj}(p)$. Then for any $r\le \min\{\operatorname{foc}(U), \frac{1}{2}\operatorname{inj}(p)\}$, any open ball $B_s(q)$ contained in $B_r(p)$ is convex, i.e,
	\begin{equation}\label{ineq-conv-foc}
	\begin{aligned}
	\operatorname{conv}(p)
	&\ge \min\{\operatorname{foc}(U), \frac{1}{2}\operatorname{inj}(p)\}\\
	&= \min\{\operatorname{foc}(U), \frac{1}{4}l(p)\}
	\end{aligned}
	\end{equation}
	where $l(p)$ is the length of a shortest non-trivial geodesic loop at $p$.
\end{theorem}

For simplicity in the remaining of the paper we use $K$ to denote the supremum of sectional curvature on $M$. Because $\operatorname{foc}(U)\ge \frac{\pi}{2\sqrt{K}}$, it is clear that (\ref{ineq-conv-foc}) implies (\ref{ineq-conv-classical-curv}). 
The equality in (\ref{ineq-conv-foc}) holds for all Riemannian homogeneous space, but (\ref{ineq-conv-foc}) may be strict on a locally symmetric space (see Example \ref{ex-homog}).

Let us recall the traditional method (see \cite{Wh, ChEb,Pe16,Karcher1970}) to find a convex neighborhood of $p$ on $M$, whose first step is to choose an open set $U$ around $p$ such that $U$ is contained in ``normal coordinates" of any point $x\in U$, or equivalently, $U$ contains no cut point of any point $x\in U$. Then any smaller neighborhood $W\subset U$ would be strongly convex once the distance function $d(\cdot,x)$ is strictly convex for any $x\in W$. Though the convexity of distance function can be guaranteed through either the focal radius or the upper sectional curvature bound $K$, the size of $U$ and hence $W$ cannot be explicitly determined unless a decay rate of injectivity radius (cf. \cite{ChLiYau,Cheeger1982}) nearby is known. In the proof of Theorem \ref{thm-conv-lowerbound}, we will show that no cut points occur in $B_{\frac{1}{2}\operatorname{inj}(p)}(p)$ for any $x$ around $p$ in distance $d(x,p)<\frac{1}{2}\operatorname{conj}(x)$. 

The second main result in this paper is that the injectivity radius admits Lipschitz decay near $p$ if $\operatorname{inj}(p)<\infty$ and is not realized by its conjugate radius.

\begin{theorem}\label{thm-inj-lipdecay} For any $p,q$ in a complete Riemannian manifold,
	\begin{equation}\label{ineq-aroundinj}
	\operatorname{inj}(q)\ge \min\{\operatorname{inj}(p), \operatorname{conj}(q)\}-d(p,q).
	\end{equation}
	In particular, if $M$ contains no conjugate points, then $\operatorname{inj}:M\to \mathbb R$ is $1$-Lipschitz, i.e.,
	$$\left|\operatorname{inj}(p)-\operatorname{inj}(q)\right|\le d(p,q).$$
\end{theorem}
It is well-known that $\operatorname{inj}(p)$ is continuous in $p$ (\cite{Kl}). And it is easy to see that $\operatorname{conj}(p)$ is continuous but not Lipschitz in $p$ in general (for example, points in a paraboloid of revolution).
\begin{remark}
	Theorem \ref{thm-conv-lowerbound} and \ref{thm-inj-lipdecay} are inspired by a recent paper \cite{Mei} by Mei, where the following lower bound on the convexity radius\footnote{It should be pointed out that convexity radius in \cite{Mei} is the concentric convexity radius defined later in this paper, but the proof in \cite{Mei} actually implies the same lower bound for convexity radius.} of $p$ and the decay of injectivity radius around $p$ were proved,
	\begin{align}
	\operatorname{conv}(p)&\ge \frac{1}{2}\min\{\frac{\pi}{\sqrt{K}}, \operatorname{inj}(p)\},\label{ineq-conv-curv}\\
	\operatorname{inj}(q)&\ge \min\{\operatorname{inj}(p), \frac{\pi}{\sqrt{K}}\}-d(p,q).\label{ineq-aroundinj-curv}
	\end{align}
	A significant improvement of (\ref{ineq-conv-curv}) from the classical local estimate  (\ref{ineq-conv-classical-curv}) on convexity radius is that no injectivity radii of nearby points around $p$ are involved.
	In Mei's proofs of (\ref{ineq-conv-curv}) and (\ref{ineq-aroundinj-curv}),
	curvature comparison with respect to $K$ is applied as the
	first step to avoid minimal geodesics going into places where $d(\cdot,p)$ is not convex (see Lemma 2.3 and Theorem 3.1 in \cite{Mei}).
	Since Theorem \ref{thm-conv-lowerbound} and Theorem \ref{thm-inj-lipdecay} are free of curvature, they cannot follow from his arguments. Moreover, according to \cite{Gulliver1975On}, there are manifolds without focal/conjugate points whose sectional curvature changes sign. Therefore (\ref{ineq-conv-foc}) (resp.  (\ref{ineq-aroundinj})) improves (\ref{ineq-conv-curv}) (resp. (\ref{ineq-aroundinj-curv})).
\end{remark}

\begin{remark}
	Both the decay rate of injectivity radius by Cheeger-Gromov-Taylor (Theorem 4.7 in \cite{Cheeger1982}) and by Cheng-Li-Yau (Theorem 1 in \cite{ChLiYau}) require the absence of conjugate points and depend crucially on the lower bound of the Ricci or sectional curvature on $M$ respectively. Hence they fail to apply as the lower bound of Ricci curvature goes to $-\infty$. However, by (\ref{ineq-aroundinj}) injectivity radius always admits Lipschitz decay rate inside $B_{\operatorname{inj}(p)}(p)$ in absence of conjugate points. In general (\ref{ineq-aroundinj}) is sharp in the sense that for any $\epsilon>0$, there are counterexamples such that 
	$$\operatorname{inj}(q)\ge \min\{\operatorname{inj}(p), \operatorname{conj}(q)\}-(1-\epsilon)d(p,q)$$
	does not hold (see Example \ref{ex-lip-sharp}).
\end{remark}
By (\ref{ineq-conv-classical-curv}) the convexity radius of $M$ admits the following lower bound,
\begin{equation}\label{ineq-conv-global-curv}
\operatorname{conv}(M)\ge \frac{1}{2}\min\{\frac{\pi}{\sqrt{K}}, \operatorname{inj}(M)\}.
\end{equation}
Let the focal radius of $M$, $\operatorname{foc}(M)$, be defined similarly as the injectivity radius of $M$.
Riemannian Manifolds with infinity focal radius are first studied in \cite{Sul}. 
 Recently Dibble improved (\ref{ineq-conv-global-curv}) to
the following equality. 
\begin{corollary}[\cite{Dib}]\label{cor-Dib-global}
	\begin{equation}\label{eq-conv-global}
	\begin{aligned}
	\operatorname{conv}(M)&=\min\{\operatorname{foc}(M),\frac{1}{2}\operatorname{inj}(M)\}
	\end{aligned}
	\end{equation}
\end{corollary}
Since (\ref{eq-conv-global}) easily follows from Theorem \ref{thm-conv-lowerbound}, (\ref{ineq-conv-foc}) can be viewed as its local version.
 
Gulliver \cite{Gulliver1975On} proved that the longest geodesic in a convex ball $B_{r}(p)$ of $M$ is the diameter of $B_{r}(p)$, provided the decay of sectional curvature away from $p$ satisfies certain convexity assumption. In particular, it was proved in \cite{Gulliver1975On} that any geodesic in $B_r(p)$ with 
\begin{equation}\label{cond-inj-curv}
0<r<\min\{\operatorname{inj}(p),\frac{\pi}{2\sqrt{K}}\}
\end{equation} has length $\le 2r$. Clearly, the conclusion fails to hold for $r>\operatorname{inj}(p)$ in general, because closed geodesics may exist in $B_r(p)$. Since the convexity of $B_r(p)$ is essential in his proof, Gulliver suspected (\cite{Gulliver1975On}) that such conclusion unlikely remains true if the curvature assumptions in (\ref{cond-inj-curv}) are replaced by weaker convexity hypothesis. However, our next theorem shows that his conclusion generally holds inside half of conjugate radius without any convexity assumption.
\begin{theorem}\label{thm-geodesic-length}
	Let $U=B_{\operatorname{inj}(p)}(p)$, and let $0<r\le \min\{\operatorname{inj}(p),\frac{1}{2}\operatorname{conj}(U)\}$.
	Then any geodesic in $B_r(p)$ has length $\le 2r$.
\end{theorem}

	In literature another convexity radius of $p$ (cf. \cite{Lafontaine1987Riemannian, IsaacChavel2006Riemannian, Pe16}) are more commonly used, which is defined to be the supremum of $r>0$ such that any open ball of radius $\le r$ centered at $p$ is convex. 
	To clarify these two concepts we call the latter \emph{concentric convexity radius} of $p$ and denote it by $\operatorname{conv_{ct}}(p)$.
	The difference between convexity radius and the concentric case is, the latter does not require that balls contained in $B_r(p)$ centered at other points than $p$ are also convex. Clearly,
	$$\operatorname{conv_{ct}}(p)\ge \operatorname{conv}(p),\text{ and }\operatorname{conv_{ct}}(M)=\operatorname{conv}(M).$$
	
	The author finds that sometimes the continuity of $\operatorname{conv_{ct}}(p)$ is confused with that of $\operatorname{conv}(p)$ in literature (eg. \cite{Azagra2007A,Hosseini2013Euler}). 
	By definition, $\operatorname{conv}(p)$ is clearly $1$-Lipschitz in $p$. In contrast, $\operatorname{conv_{ct}}(p)$ may be not continuous in $p$, and it may well happen that $\operatorname{conv_{ct}}(p)>\operatorname{conv}(p)$. 
	In this paper we will clarify these two different concepts and give explicit counterexamples on the discontinuity of $\operatorname{conv_{ct}}$. We state it as the following theorem.
	
	\begin{theorem}\label{thm-counterexample}
		Any smooth manifold $M$ of dimension $\ge 2$ admits a Riemannian metric such that there is a sequence $z_i$ of points converging to $z$ with $\operatorname{conv_{ct}}(z_i)<2$ and $\operatorname{conv_{ct}}(z)> 10$.
	\end{theorem}
	By the continuity of convexity radius, in the above theorem we have  $\operatorname{conv_{ct}}(z)>\operatorname{conv}(z)$.
	The Riemannian metrics constructed by Gulliver in \cite{Gulliver1975On} satisfies the conclusion of Theorem \ref{thm-counterexample}. After this paper was put on the arXiv, the author was informed that Dibble had independently constructed examples \cite{Dib2} to show the discontinuity of the concentric convexity radius.

	We now give a sketched proof of Theorem \ref{thm-conv-lowerbound}. 
	Recall that Klingenberg's Lemma (cf. \cite{Kl,ChEb}, or Lemma 3.1 in Section 3) says that if $q$ is a nearest cut point of $p$ such that $q$ is not conjugate to $p$ along any minimal geodesic connecting them, then there are exactly two minimal geodesics between $p$ and $q$ which form a geodesic loop at $p$ formed by passing through $q$. It directly follows that the injectivity radius and conjugate radius satisfy
	\begin{equation}\label{eq-inj-pt}
	\operatorname{inj}(p)=\min\{\frac{1}{2}l(p), \operatorname{conj}(p)\}.
	\end{equation}
	 We find that what happens for convexity radius and focal radius is  similar to the above. In proving (\ref{ineq-conv-foc}) the key technical tool is a \emph{generalized} version of Klingenberg's lemma, as follows. 
	 
	 \begin{lemma}[Generalized Klingenberg's lemma, 
	 \cite{Xu2016On, Innami2012The}]\label{lem-gene-kling}
	 Let $p,q$ be any two points in $M$, and $x_0$ be a cut point of $q$ which minimizes the perimeter function $P_{p;q}=d(p,\cdot)+d(\cdot,q)+d(p,q)$ in the set $\operatorname{Cut}(q)$ of cut points of $q$. If $x_0$ is not a conjugate of $q$, then up to a reparametrization there are at most $2$ minimal geodesics from $p$ to $x_0$, and each of them admits an extension at $x_0$ that is a minimal geodesic from $x_0$ to $q$. 
	 \end{lemma}
	 
	The proofs of Theorem \ref{thm-conv-lowerbound} and \ref{thm-inj-lipdecay} are rather simple once the generalized version of Klingenberg's lemma is known.
	For Theorem 1.1 we argue by contradiction.
	Let $0<r\le \frac{1}{2}\operatorname{inj}(p)$ such that for any point $x\in U=B_{\operatorname{inj}(p)}(p)$, $r\le \operatorname{foc}(x)$. By elementary facts on focal radius, the distance function $d(x,\cdot)$ is convex inside $B_r(p)$ where it is smooth. If $B_r(p)$ is not strongly convex, then by a standard closed-and-open argument (see the proof of Proposition 2.1) it is not difficult to deduce that there are $q,z\in B_r(p)$ with $z\in \operatorname{Cut}(q)$. Let $x_0$ be a minimum point in $\operatorname{Cut}(q)$ of the perimeter function $P_{p;q}(x)$. Then $$d(x_0,p)\le \frac{1}{2}P_{p;q}(x_0)\le 2r <\operatorname{inj}(p),$$
	which implies that $x_0\in U$. Because $d(x_0,q)$ is strictly less than $2r\le 2 \operatorname{foc}(U)$, $x_0$ is not a conjugate point of $q$. By the generalized Klingenberg's Lemma, we conclude that there is  another geodesic from $p$ to $q$ passing through $x_0$ which lies in $U$, a contradiction. The proof of Theorem \ref{thm-inj-lipdecay} is similar.
	
	\begin{remark}
		The global version of Lemma \ref{lem-gene-kling} was first established in \cite{Innami2012The}, where the same idea as above was applied to conclude the existence of pole points whose injectivity radius is infinite. Here we are using the local version developed in \cite{Xu2016On}. We will give a quick review and a simplified proof of Lemma \ref{lem-gene-kling} in Section 3. It turns out that the condition on conjugate points in Theorem \ref{thm-inj-lipdecay} can be further weakened slightly to ``totally conjugate" points defined in \cite{Xu2016On} (see the last part of Section 3).
	\end{remark}
		
	We also have similar lower bounds of the concentric convexity radius in terms of $\operatorname{foc}(p)$ and conjugate radii around $p$ (see Theorem \ref{thm-convpt-exfoc}). Among other things the existence of convex neighborhoods has been widely used in various problems of differential geometry (see \cite{Ber}). Further applications of (\ref{ineq-conv-foc}) on convex radius and the decay rate of injectivity radius (\ref{ineq-aroundinj}) would be given in separate papers such as \cite{HuXu}.
	
	The rest of the paper is organized as follows. Based on some basic properties of focal point and focal radius from \cite{Sul, Dib}, a general two-sided bound on (concentric) convexity radius in terms of focal radius is given in Section 2. In Section 3 we first review the generalized Klingenberg's Lemma together with a simple proof, and then prove main theorems \ref{thm-conv-lowerbound}-\ref{thm-geodesic-length}. Theorem \ref{thm-counterexample} is proved in the last Section 4, and other examples mentioned above are given there.
	
	\textbf{Acknowledgement:} The author would like to thank Xiaochun Rong for his support, criticisms and suggestions which have improved the organization of this paper, and also thank Jiaqiang Mei for kindly sharing his results and preprint. The author is also grateful to James Dibble and Hongzhi Huang for their helpful comments. The author is supported partially by NSFC
	Grant 11401398 and by a research fund from Capital Normal University.
	
	\section{Preliminaries on Convexity Radius and Focal Radius}
	Let us recall that the focal radius of $p$, $\operatorname{foc}(p)$, is defined by
	$$\begin{aligned}
	\operatorname{foc}(p)=\min\{T: \exists \text{ non-trivial Jacobi field $J$ along a unit-speed }\\ 
	\text{ geodesic $\gamma$ such that $\gamma(0)=p$, $|J|'(T)=0$}\}.
	\end{aligned}
	$$
	We define the \emph{extended focal radius} of $p$, denoted by $\operatorname{foc_e}(p)$, to be 
	$$\begin{aligned}
	\operatorname{foc_e}(p)=\inf\{T: \exists \text{ non-trivial Jacobi field $J$ along a unit-speed }\\ 
	\text{ geodesic $\gamma$ such that $\gamma(0)=p$, $|J|'(T)<0$}\}.
	\end{aligned}$$
	That is, the extended focal radius is the supremum of the radius of open balls centered at the origin of $T_pM$ such that for each tangent vector $v$ in the ball and any normal Jacobi field $J$ along the radial geodesic $\exp_ptv$ with $J(0)=0$, $\frac{d}{dt}|J(t)|\ge 0$ $(0<t\le 1)$ holds.
	Let $R_c(p)$ be the \emph{cut-decay radius} of $p$ defined by
	$$R_c(p)=\sup\{r: \forall x\in B_r(p), B_r(p)\cap\operatorname{Cut}(x)=\emptyset\},$$
	where $\operatorname{Cut}(x)$ is the set of cut points of $x$ in $M$. It is clear that $R_c(p)\ge \frac{1}{2}\operatorname{inj}(M)$.
	
In order to clarify the difference between the convexity radius and the concentric one, we now give general descriptions of them in terms of the (extended) focal radius, by reformulating some standard facts in Riemannian geometry into a two-sided bound of (concentric) convexity radius of a point $p$ in a complete Riemannian manifold $M$.
	\begin{proposition}[\cite{KARCHER1968,Dib}, cf. \cite{Gromoll1968Riemannsche}]\label{prop-conv-focal-cutdecay}
		$ $
		\begin{enumerate}[(\ref{prop-conv-focal-cutdecay}.1)]
			\item $\operatorname{conv_{ct}}(p)=\min\{\operatorname{foc_e}(p), R_c(p)\},$
			\item $\min\{\operatorname{foc}(B_{R_c(p)}(p)), R_c(p)\}\le \operatorname{conv}(p)\le \min\{\operatorname{foc}(p),R_c(p)\},$
		\end{enumerate}
	\end{proposition}
(\ref{prop-conv-focal-cutdecay}.1) is a consequence of \cite{KARCHER1968} (cf. Proposition 1 in \cite{Alexander1978Local}). (\ref{prop-conv-focal-cutdecay}.2) can be deduced from the Morse index theorem (\cite{Sul}, cf. \cite{Gromoll1968Riemannsche,ChEb}) and results in \cite{Dib}. Proposition \ref{prop-conv-focal-cutdecay} will be used in the proofs of the main theorems \ref{thm-conv-lowerbound} and \ref{thm-counterexample}. For completeness we will give a direct proof of Proposition \ref{prop-conv-focal-cutdecay} in this section.

\begin{remark}
It is easy to find examples on von Mangoldt's surfaces of revolution (\cite{Tanaka}) such that both $\frac{R_c(p)}{\operatorname{inj}(p)}$
and $\frac{\operatorname{foc_e}(p)}{\operatorname{inj}(p)}$
can be arbitrary small. Examples that $\frac{\operatorname{foc}(M)}{\operatorname{inj}(M)}$ can be arbitrary small was constructed in \cite{Dib}.	
\end{remark}

	Let us first review some basic facts from \cite{Sul} about focal points.
	Let $\alpha:(a,b)\to M$ be a minimal geodesic. $p$ is called a \emph{focal point} of $\alpha$ along a geodesic $\gamma_{s_0}$ if there is a variation $\bar \gamma(s,t):(s_0-\epsilon,s_0+\epsilon)\times [0,1]\to M$ of $\gamma_{s_0}$ through geodesics $\bar \gamma(s,\cdot)$ such that
	$\bar \gamma(s,0)=\alpha(s)$, $\bar\gamma(s_0,t)=\gamma_{s_0}(t)$, $\frac{\partial}{\partial t}\bar \gamma(s,0)\perp \alpha'(s)$ and $\bar \gamma(s_0,1)=p$.
	So focal points of $\alpha$ are just critical points of $$\exp:\nu (\alpha)\to M , \quad (s,v)\in\nu(\alpha)\mapsto\exp_{\alpha(s)}(v),$$ where $\nu(\alpha)$ is the normal bundle of $\alpha$. It is an elementary fact of Jacobi field that the focal radius of $p$ describes how far away from $p$ there exists a geodesic whose focal points containing $p$.
	\begin{lemma}[\cite{Sul}]\label{lem-focal-radi-point}
		A point $p$ is a focal point of $\alpha$ along a geodesic $\gamma:[0,1]\to M$ if and only if $\gamma(0)=\alpha(s_0)$, $\gamma(1)=p$ for some $s_0\in(a,b)$ and there is a perpendicular Jacobi field $J$ along $\gamma$ such that $J(0)=\alpha'(s_0)$, $J(1)=0$ and $\nabla_{\gamma'}J(0)\perp J(0)$.
	\end{lemma}
\begin{proof}
	Let $p$ be a focal point of $\alpha$ and $\bar \gamma(s,t)$ be the above variation though geodesics associated to $p$. Let $T=\frac{\partial}{\partial t}\bar \gamma$ and $V=\frac{\partial}{\partial s}\bar \gamma$. Then $V\perp T$ and  
	$$\left< \nabla_{T}V,V \right>(s,0)=\left<\nabla_{V}T,V\right>(s,0)=-\left<T,\nabla_VV\right>(s,0)=0,$$ which implies that $\gamma=\bar\gamma_{s_0}$ and $J(t)=V(s_0,t)$ satisfies the requirement of Lemma \ref{lem-focal-radi-point}.
	
	Conversely, let $\gamma$ be a geodesic and $J$ a perpendicular Jacobi field along $\gamma$ in Lemma \ref{lem-focal-radi-point}, and let $T(s)$ be a vector field perpendicular to $\alpha$ such that $T(s_0)=\gamma'(0)$ and $\nabla_{\alpha'}T(s_0)=\nabla_{\gamma'}J(0)$. Then the variation $\bar{\gamma}(s,t)=\exp_{\alpha(s)}tT(s)$ satisfies $\frac{\partial}{\partial s}\bar{\gamma}(s_0,t)=J(t)$. Hence $p$ is a focal point of $\alpha$ along $\gamma$.
\end{proof}
	By Lemma \ref{lem-focal-radi-point}, we say that $p$ is \emph{focal to} $\alpha$ at $\alpha(0)$ along a geodesic $\gamma:[0,1]\to M$ if $p$ is a focal point of $\alpha$ along $\gamma(1-t)$.  Now it is clear that $\operatorname{foc}(p)\le R<\infty$ if and only if $p$ is focal to a minimal geodesic $\alpha$ along a unit-speed geodesic $\gamma:[0,R]\to M$. 
	
	The following elementary facts on (extended) focal radius are used in proving Proposition 2.1 and Theorem \ref{thm-counterexample}.
	\begin{lemma}\label{lem-elemen-focal} $ $
		\begin{enumerate}[(\ref{lem-elemen-focal}.1)]
			\item
			$\operatorname{foc}(p)\le\operatorname{foc_e}(p)\le \operatorname{conj}(p)$, and if $\operatorname{conj}(p)<\infty$, then $\operatorname{foc_e}(p)<\operatorname{conj}(p)$;
			\item Let $0<R<\operatorname{conj}(p)$ and $\exp_p^*g$ be the pull-back metric on the tangent space $T_pM$ from $M$. Then $r^2:(T_pM,\exp_p^*g)\to \mathbb R$, $r^2(v)=|v|^2$ is convex (respectively, strictly convex) on the open ball $B_{R}(0)$ if and only if $R\le\operatorname{foc_e}(p)$ (resp. $R\le \operatorname{foc}(p)$).
			\item $\operatorname{foc_e}(p)$ is upper semi-continuous and $\operatorname{foc}(p)$ is lower semi-continuous, i.e,
			$$\limsup_{p_i\to p}\operatorname{foc_e}(p_i)\le \operatorname{foc_e}(p),\quad\text{and}\quad
			\liminf_{p_i\to p}\operatorname{foc}(p_i)\ge \operatorname{foc}(p).$$
		\end{enumerate}	
	\end{lemma}
	\begin{proof}
		(\ref{lem-elemen-focal}.1) is clear for any non-trivial Jacobi field vanishing at the start point admits $|J|'=0$ or $|J|'<0$ before $|J|$ becomes $0$ again. 
		
		(\ref{lem-elemen-focal}.2) follows from the fact that $\operatorname{Hess} r^2(J,J)=r\left\langle J, J\right\rangle'$ for any Jacobi field $J$ in the definition of $\operatorname{foc_e}(p)$.
		
		By definition, it is clear that (\ref{lem-elemen-focal}.3) holds.
	\end{proof}

	It will be seen from the proof of Theorem \ref{thm-counterexample}  that it may well happen  
	$$\limsup_{p_i\to p}\operatorname{foc_e}(p_i)< \operatorname{foc_e}(p)\quad\text{and}\quad \liminf_{q_i\to q}\operatorname{foc}(q_i)> \operatorname{foc}(q).$$
	In order to derive the upper bound in (\ref{prop-conv-focal-cutdecay}.2) and the discontinuity of concentric convexity radius in Theorem \ref{thm-counterexample}, the most important property of focal radius is the following.
	\begin{lemma}[\cite{Dib}]\label{lem-converg-foc-radi}
		If $\operatorname{foc}(p)<\infty$, then there is a sequence of $p_i$ converging to $p$ such that $$\lim_{i\to\infty}\operatorname{foc_e}(p_i)=\lim_{i\to\infty}\operatorname{foc}(p_i)=\operatorname{foc}(p).$$
	\end{lemma}

	\begin{proof}
		It suffices to show the case that $R=\operatorname{foc}(p)<\operatorname{inj}(p)$, for this is always true after lifting $p$ as well as the geodesic and Jacobi field to $T_pM$.
		According to Lemma \ref{lem-focal-radi-point}, there are a unit-speed minimal geodesic $\alpha:(-a,a)\to M$, a geodesic $\gamma:[0,1]\to M$ of length $R$ from $\alpha(0)$ to $p=\gamma(1)$, and a Jacobi field $J$ along $\gamma$ such that $p$ is a focal of $\alpha$ along $\gamma:[0,1]\to M$ with $\gamma(0)=\alpha(0)$, $\gamma(1)=p$, $J(0)=\alpha'(0)$ and $J(1)=0$. 
		
		Let $\epsilon>0$ small enough such that $R(1+\epsilon)<\operatorname{inj}(\alpha(0))$ and $\epsilon<\operatorname{inj}(q)$ for any $q$ nearby $p$.
		We claim that
		\begin{claim*}
			$d(\gamma(1+\epsilon),\alpha(s))$ is not a convex function of $s$.
		\end{claim*}
		Clearly it follows from the claim that  $\operatorname{foc_e}(\gamma(1+\epsilon))\le R(1+\epsilon)$. Hence by (\ref{lem-elemen-focal}.3) $$\lim_{\epsilon\to0}\operatorname{foc_e}(\gamma(1+\epsilon))=R,$$
		and thus 
		$$\lim_{\epsilon\to0}\operatorname{foc}(\gamma(1+\epsilon))=\operatorname{foc}(p).$$
				
		The claim is a consequence of the Morse index theorem, which implies that any geodesic normal to $\alpha$ stops realizing the distance to $\alpha$ after passing a focal point. 
		Here we give another direct proof, whose idea is similar to Lemma 5.7.8 in \cite{Pe16}. Let us argue by contradiction. Let $$r_\epsilon(x)=d(x,\gamma(1+\epsilon)), \quad 
		r_\alpha(x)=\inf_{s\in(-a,a)}d(x,\alpha(s))$$ be the distance functions to $\gamma(1+\epsilon)$ and $\alpha$ respectively, and let 
		$e(x)=r_{\alpha}(x)+r_{\epsilon}(x)$ be an auxiliary function. If $r_\epsilon\circ\alpha(s)$ is convex, then $$d_\alpha(\gamma(1+\epsilon))=r_\epsilon(\alpha(0))=R(1+\epsilon).$$ 
		It follows that $e(\gamma(t))=R(1+\epsilon)$.
		By the triangle inequality, for any $x$ around $\gamma(1)$, $$e(x)=r_{\alpha}(x)+r_{\epsilon}(x)\ge d_\alpha(\gamma(1+\epsilon))=e(\gamma(t)),$$
		which implies $\left.\operatorname{Hess}e\right|_{\gamma(t)}\ge 0$.
		
		In order to meet a contradiction, in the following we will show that  for any $0<t\to 1^-$,
		$$\left.\operatorname{Hess}e\right |_{\gamma(t)}(J,J)<0.$$
		Let $T=\gamma'(t)$. Since $\nabla_TJ(1)\neq 0$, by the Taylor extension of $|J|^2$ and $\left\langle \nabla_TJ,J\right\rangle$ at $t=1$,
		$$|J|^2(t)=(t-1)^2 \left\langle \nabla_TJ,\nabla_TJ\right\rangle+o\left((t-1)^2\right)$$
		$$\left\langle \nabla_TJ,J\right\rangle=(t-1)\left\langle \nabla_TJ,\nabla_TJ\right\rangle+o(t-1)$$
		one has
		$$\frac{\left\langle\nabla_{T}J,  J \right\rangle}{\left\langle J,J\right\rangle}(t)\to -\infty, \quad\text{as $t\to 1^-$}.$$
		Because 
		$$\operatorname{Hess}r_\alpha(J,J)=\frac{1}{R}\left\langle\nabla_{J}T,  J \right\rangle=\frac{1}{R}\left\langle\nabla_{T}J,  J \right\rangle(t)$$ and
		$\operatorname{Hess}r_\epsilon$ at $\gamma(t)$ is bounded above by a constant that depends on $\epsilon$ and sectional curvature of $M$,
		$$\left.\operatorname{Hess}e\right|_{\gamma(t)}(\frac{J}{|J|},\frac{J}{|J|})
		= \frac{1}{R}\frac{\left\langle\nabla_{T}J,  J \right\rangle}{\left\langle J,J\right\rangle}(t) + \left.\operatorname{Hess}r_{\epsilon}\right|_{\gamma(t)}(\frac{J}{|J|},\frac{J}{|J|})\to -\infty
		$$
		as $t\to 1^-$. 
	\end{proof}
	
	We are now ready to prove Proposition \ref{prop-conv-focal-cutdecay}.
	\begin{proof}[Proof of Proposition 2.1]~\newline
		\indent Let us first prove (2.1.1). Because the cut points connected by at least two minimal geodesics from $x$ is dense in $\operatorname{Cut}(x)$ (cf. \cite{Wolter1979Distance,Bishop1977Decomposition,Kl}), it is clear that $\operatorname{conv_{ct}}(p)\le R_c(p)$. By (2.4.2), $\operatorname{conv_{ct}}(p)\le \operatorname{foc_e}(p)$.
		
		Next, let us show $\operatorname{conv_{ct}}(p)\ge \min\{\operatorname{foc_e}(p), R_c(p) \}$, that is, for any $0<r<\min\{\operatorname{foc_e}(p), R_c(p) \}$, the open ball $B_r(p)$ is convex. For any $q\in B_r(p)$, let us consider the set 
		$$U_q=\{x\in B_r(p) : \text{ minimal geodesic between $x$ and $q$ in $M$ lies in $B_r(p)$}\}.$$ It is clear that $U_q\neq \emptyset$ is open.
		By (2.4.1) and (2.4.2), $d(p,\cdot)$ is convex in $B_r(p)$ and thus $d(p,\gamma(t))$ is a convex function for any minimal geodesic $\gamma$ in $B_r(p).$ Let $\gamma_i$ be a converging sequence of minimal geodesics such that each $\gamma_i$ lies in $B_r(p)$. Then $d(p,\gamma(t))$, the limit of $d(p,\gamma_i(t))$, is still convex. Hence $\gamma$ lies in $B_r(p)$ as long as its endpoints are in $B_r(p)$. Since the minimal geodesic between any two points in $B_r(p)$ is unique, it follows that $U_q$ is closed. By the connectedness of $B_r(p)$, $U_q=B_r(p)$. Therefore $B_r(p)$ is convex.
		
		Now it is easy to see (2.1.2) holds. Indeed, the upper bound in (2.1.2) follows from the continuity of $\operatorname{conv}$ and Lemma \ref{lem-converg-foc-radi}. It is clear that for any $x\in M$ and any $r>0$ such that $B_r(x)\subset B_{\min\{\operatorname{foc}(B_{R_c(p)}(p)), R_c(p)\}}(p)$,  $R_c(x)\ge r$. By (2.1.1) $B_r(x)$ is convex. Hence the lower bound in (2.1.2) holds.
	\end{proof}
	
	Note that it happens that $\operatorname{foc}(p)<\operatorname{foc_e}(p)$ (see Theorem \ref{thm-counterexample-euclidean}) and $\operatorname{foc}(p)>\frac{1}{2}\operatorname{conj}(p)$. Such inequalities, however, globally won't happen and the followings hold.
	\begin{lemma}[\cite{Dib}]\label{lem-global-efoc-conj} $ $
		\begin{enumerate}[(\ref{lem-global-efoc-conj}.1)]
		\item
		Let $U$ be a open set of $M$, then
		$\operatorname{foc_e}(U)=\operatorname{foc}(U)$.
		\item
		$\operatorname{foc}(M)\le \frac{1}{2}\operatorname{conj}(M).$
	\end{enumerate}
	\end{lemma}
	\begin{proof}
	By Lemma \ref{lem-converg-foc-radi}, it is clear that for any open set $U\subset M$,
	$$\operatorname{foc_e}(U)=\inf_{p\in U} \{\operatorname{foc_e}(p)\}=\operatorname{foc}(U).$$
	If $p$ and $q$ are conjugate along a geodesic $\gamma$, then 
	by definition $$\operatorname{foc}(p)+\operatorname{foc}(q)\le \operatorname{length}(\gamma).$$ By approximating $\operatorname{conj}(M)$ by $p,q$ and $\gamma$, it follows that (see \cite{Dib})
	$$\operatorname{foc}(M)\le \frac{1}{2}\operatorname{conj}(M).$$
	\end{proof}
	
	Assuming Theorem \ref{thm-conv-lowerbound}, we give a proof of Corollary \ref{cor-Dib-global}.
	\begin{proof}[Proof of Corollary \ref{cor-Dib-global}] $ $\\
	\indent Firstly, by Theorem \ref{thm-conv-lowerbound} we have
	$\operatorname{conv}(M)\ge \min\{\operatorname{foc}(M),\frac{1}{2}\operatorname{inj}(M)\}.$
	
	Secondly, for any $\epsilon>0$, by the continuity of $\operatorname{inj}(p)$, there is $p\in M$ such that $\operatorname{inj}(p)<\operatorname{inj}(M)+\epsilon$. Let $q$ be the midpoint of a minimal geodesic at $p$ which realizes the distance from $p$
	to its cut points, then $R_c(q)\le \frac{1}{2}\left(\operatorname{inj}(M)+\epsilon\right)$.
	Now by Proposition \ref{prop-conv-focal-cutdecay}, $$	\operatorname{conv}(M)\le \operatorname{conv}(q)\le R_c(q)\le \frac{1}{2} \left(\operatorname{inj}(M)+\epsilon\right),$$
	and $\operatorname{conv}(M)\le \operatorname{foc}(M)$. Therefore 
	$\operatorname{conv}(M)\le  \min\{\operatorname{foc}(M),\frac{1}{2}\operatorname{inj}(M)\}.$
	\end{proof}

	\section{Generalized Klingenberg's Lemma}
	In this section we first summarize recent developments 
	(\cite{Xu2016On,Innami2012The}) on a generalized version of Klingenberg' lemma with simplified proofs presented, then we prove Theorem \ref{thm-conv-lowerbound}, \ref{thm-inj-lipdecay} and \ref{thm-geodesic-length}. Let us first recall Klingenberg's lemma, which has been a basic fact in Riemannian geometry. 
	\begin{lemma}[Klingenberg, \cite{Kl,ChEb}]\label{lem-orig-kling}
		Let $z$ be a local nearest cut point of $p$ in $\operatorname{Cut}(p)$. If $z$ is not a singular point of $\operatorname{exp}_p$, then there are exactly two minimal geodesics connecting $p$ and $z$ that form a geodesic loop at $p$.
	\end{lemma}

 In order to conclude a nontrivial geodesic loop, Lemma \ref{lem-orig-kling} requires the existence of at least two minimal geodesics along which $p$ and $z$ are not conjugate (see \cite{Kl,ChEb,Pe16}). In \cite{Xu2016On} we showed that assuming only one of such minimal geodesic it is also enough to get the same conclusion. Therefore Lemma \ref{lem-orig-kling} was improved to the following form.
	\begin{lemma}[Improved Klingenberg's Lemma, \cite{Xu2016On}]\label{lem-klingenberg-improved}
		Let $z$ be a local nearest point of $p$ in $\operatorname{Cut}(p)$. If there is a minimal geodesic $\alpha:[0,1]\to M$ from $p=\alpha(0)$ to $z=\alpha(1)$ along which they are not conjugate, then up to a reparametrization there are exactly two minimal geodesics between $p$ and $z$, and they form a geodesic loop at $p$. 
	\end{lemma}
Because Lemma \ref{lem-klingenberg-improved} is a special case of the generalized version for two points and their proofs enjoy similar ideas, we give a short proof of the simple case before going into the general one.

	\begin{proof}
		Let us consider the function $F:M\to \mathbb{R}$, $F(x)=2d(x,p)$. Then $z$ is a local minimum of $\left.F\right|_{\operatorname{Cut}(p)}$. Because $p,z$ are not conjugate along $\alpha$, there is another minimal geodesic $\beta:[0,1]\to \mathbb{R}$ from $p$ to $z$. For any $t\in (0,1)$, let $\gamma_t:[0,1]\to M$ be a minimal geodesic connecting $\alpha(t)$ and $\beta(t)$. Then 
		\begin{align*}
		F(\gamma_t(s))&= d(\gamma_t(s),p)+d(\gamma_t(s),p)\\
		&\le d(\gamma_t(s),\alpha(t))+d(\alpha(t),p)+d(\gamma_t(s),\beta(t))+d(\beta(t),p)\\
		&=d(\alpha(t),\beta(t))+d(\alpha(t),p)+d(\beta(t),p)\\
		&\le d(\alpha(t),z)+d(\beta(t),z)+d(\alpha(t),p)+d(\beta(t),p) \tag{$*$}\\
		&\le 2d(p,z)\\
		&=F(z)
		\end{align*}
		Note that $(*)$ is strict if and only if tangent vector $\alpha'(1)\neq -\beta'(1)$. Assume that $\alpha'(1)\neq -\beta'(1)$, then $F(\gamma_t(s))<F(z)$, which implies that there is $\epsilon_1>0$ such that for any $1-\epsilon_1<t<1$, $\gamma_t(s)$ is not a cut point of $p$ ($0\le s\le 1$). Take open neighborhoods $U$ of $z$ in $M$ and $W$ of $\alpha'(0)$ in $T_pM$ such that $\left.\exp_p\right|_W:W\to U$ is a diffeomorphism. Then there is $\epsilon_2>0$ such that for $1-\epsilon_2<t<1$, $\gamma_t$ lies in $U$. For $1-\min\{\epsilon_1,\epsilon_2\}<t<1$, let $\tilde\gamma_t$ be the lift of $\gamma_t$ in $W\subset T_pM$. Because the curve $\gamma_t$ has no intersection with $\operatorname{Cut}(p)$, the endpoint of $\tilde\gamma_t$ satisfies 
		$$\tilde\gamma_t(1)=t\cdot \beta'(0)\in W, \quad\text{for any $1-\min\{\epsilon_1,\epsilon_2\}<t<1$}.$$ However, $\beta'(0)\not\in W$, which implies that there is $\epsilon_3>0$ such that $$t\cdot\beta'(0)\not\in W, \quad \text{for any  $1-\epsilon_3<t\le 1$}.$$ We meet a contradiction for any  $1-\min\{\epsilon_1,\epsilon_2,\epsilon_3\}<t<1$.
		
		In particular, the above arguments imply that any minimal geodesic $\beta:[0,1]\to M$ other than $\alpha:[0,1]\to M$ satisfies $\beta'(0)=-\alpha'(0)$. Hence the minimal geodesic between $p$ and $z$ other than $\alpha$ is unique.
	\end{proof}
	As first observed in \cite{Innami2012The}, Klingenberg's lemma can be generalized to the case of two points. In the following we reformulate the local version in \cite{Xu2016On} into the form of Lemma \ref{lem-gene-kling} in the introduction. Let $p,q\in M$ be two fixed points. For any $x\in M$, let $P_{p;q}(x)$ be the perimeter of geodesic triangle $\triangle(pqx)$, i.e.,  $$P_{p;q}:M\to \mathbb{R},
	\quad P_{p;q}(x)=d(p,x)+d(x,q)+d(p,q).$$
	\begin{lemma}[Generalized Klingenberg's Lemma, \cite{Xu2016On,Innami2012The}]\label{lem-klingenberg-generalized}
		Let $x_0\in \operatorname{Cut}(q)$ be a local minimum point of the perimeter function $\left.P_{p;q}\right|_{\operatorname{Cut}(q)}$ in $\operatorname{Cut}(q)$. If there is a minimal geodesic from $q$ to $x_0$ along which $q$ and $x_0$ are not conjugate, then up to a reparametrization, there are at most $2$ minimal geodesics from $p$ to $x_0$, and each of them admits an extension at $x_0$ that is a minimal geodesic from $x_0$ to $q$. More precisely, if $p\neq x_0$ and $\alpha_1:[0,1]\to M$ is a minimal geodesic from $q=\alpha_1(0)$ to $x_0=\alpha_1(1)$ such that $d\exp_q$ is non-degenerated at $\alpha_1'(0)\in T_qM$, then
		\begin{enumerate}[(\ref{lem-klingenberg-generalized}.1)]
			\item either there is a unique minimal geodesic $\beta_1:[0,1]\to M$ from $p$ to $x_0$, and it satisfies $$\frac{1}{|\beta_1'(1)|}\cdot \beta_1'(1)=-\frac{1}{|\alpha_1'(1)|}\cdot \alpha_1'(1);$$
			\item or there are a unique minimal geodesic $\beta_1:[0,1]\to M$ from $p$ to $x_0$ and exactly two minimal geodesics $\alpha_1,\alpha_2:[0,1]\to M$ connecting $q$ and $x_0$, and they satisfy
			$$\frac{1}{|\beta_1'(1)|}\cdot \beta_1'(1)=-\frac{1}{|\alpha_2'(1)|}\cdot \alpha_2'(1);$$
			\item or there are exactly two minimal geodesics $\beta_i:[0,1]\to M$ from $p$ to $x_0$ and exactly two minimal geodesic $\alpha_i:[0,1]\to M$ from $q$ to $x_0$, and up to a permutation of $\beta_i$ ($i=1,2$) they satisfy
			$$\frac{1}{|\beta_1'(1)|}\cdot \beta_1'(1)=-\frac{1}{|\alpha_1'(1)|}\cdot \alpha_1'(1), \quad \frac{1}{|\beta_2'(1)|}\cdot \beta_2'(1)=-\frac{1}{|\alpha_2'(1)|}\cdot \alpha_2'(1).$$ 
		\end{enumerate}
	\end{lemma}
If $p=x_0$ in Lemma \ref{lem-klingenberg-generalized}, then the conclusion holds trivially after viewing $p$ as a minimal geodesic from $p$ to itself. If $p=q$, then Lemma \ref{lem-klingenberg-generalized} coincides with Lemma \ref{lem-klingenberg-improved}.
	\begin{proof} In the following we assume that $p\neq x_0$.
		Since $\exp_q$ is non-singular at $\alpha_1'(0)$, there are at least two minimal geodesics connecting $q$ and $x_0$. We claim that
		
		\begin{claim*}
			For any minimal geodesic $\beta_1:[0,1]\to M$ from $\beta_1(0)=p$ to $\beta_1(1)=x_0$, if $\beta_1$ does not coincide with the extension of $\alpha_1$ at $x_0$, i.e.,
			$$\alpha_1'(1)\neq -\frac{d(q,x_0)}{d(p,x_0)}\beta_1'(1).$$ Then for any other minimal geodesic $\alpha_2:[0,1]\to M$ from $q$ to $x_0$ coincides with the extension of $\beta_1$ at $x_0$, i.e., $$\alpha_2'(1)=-\frac{d(q,x_0)}{d(p,x_0)}\cdot \beta_1'(1).$$
			In particular, there are exactly two minimal geodesics $\alpha_1$ and $\alpha_2$ between $q$ and $x_0$.
		\end{claim*}  
		
		In order to prove the claim, let us argue by contradiction. Let $\alpha_2:[0,1]\to M$ be another minimal geodesic from $q$ to $x_0$ such that both $\alpha_i'(1)\neq - \frac{d(q,x_0)}{d(p,x_0)}\cdot \beta_1'(1)$ for $i=1,2$. Let $\gamma_{i,t}:[0,1]\to M$ be a minimal geodesic from $\beta_1(t)$ to $\alpha_i(t)$. 
		Let us choose an open neighborhood $U$ of $x_0$ such that $P_{p;q}(x_0)$ is minimizing in $P_{p;q}(\operatorname{Cut}(p)\cap U)$ and there is an open neighborhood $W$ of $\alpha_1'(1)$ in $T_qM$, restricted on which $\left.\exp_q\right|_W:W\to U$ is a diffeomorphism. For any $0<t<1$ sufficient close to $1$,  $\gamma_{i,t}$ lies in $U$, hence has a lift $\tilde\gamma_{i,t}$ in $W$. 
		
		Since $\alpha_i'(1)\neq-\frac{d(q,x_0)}{d(p,x_0)}\cdot \beta_1'(1)$, the same argument as in the proof of Lemma \ref{lem-klingenberg-improved} shows that
		\begin{align*}
		P_{p;q}(\gamma_{i,t}(s))< P_{p;q}(x_0). \tag{$**$}
		\end{align*}
		 The fact that $x_0$ is a minimum point of $\left.P_{p;q}\right|_{\operatorname{Cut}(q)\cap U}$ and $(**)$ implies that none of $\gamma_{i,t}$ has intersection with $\operatorname{Cut}(q)$. It follows that the endpoint of $\tilde\gamma_{i;t}$ satisfies
		 $$t\cdot \alpha_i'(0)=\tilde\gamma_{i,t}(1)\in W, \quad i=1,2.$$
		 At the same time, $\alpha_2'(0)\not\in W$ implies that,  $$t\cdot \alpha_2'(0)\not\in W, \quad \text{when $t$ close to $1$.}$$ So we meet a contradiction.
		 
		 By the claim, if there is a a unique minimal geodesic $\beta_1$ from $p$ to $x_0$, then either its extension at $x_0$ coincides with $\alpha_1$, or there are exactly two minimal geodesics $\alpha_{i}$ $(i=1,2)$ from $q$ to $x_0$ such that $\alpha_2$ coincides with the extension of $\beta_1$ at $x_0$.
		 
		 Now let us assume that there is another minimal geodesic $\beta_2:[0,1]\to M$ from $p$ to $x_0$ other than $\beta_1$. Then it is clear that at lease one of $\beta_j$, say $\beta_2$, does not coincide with the extension of $\alpha_1$ at $x_0$. By applying the claim to $\beta_2$, we see that there are exactly two minimal geodesics $\alpha_1,\alpha_2:[0,1]\to M$ from $q$ to $x_0$ such that $\beta_2'(1)=-\frac{d(p,x_0)}{d(q,x_0)}\cdot \alpha_2'(1).$
		 In this case $\alpha_1$ must coincides with the extension of $\beta_1$ at $x_0$. 
		 
		 Summarizing the above, we have proved that one and only one of the three cases (\ref{lem-klingenberg-generalized}.1)-(\ref{lem-klingenberg-generalized}.3) can happen.
	\end{proof}
	Now we are ready to prove Theorem \ref{thm-conv-lowerbound}, Theorem \ref{thm-inj-lipdecay} and Theorem \ref{thm-geodesic-length}.
	\begin{proof}[Proof of Theorem \ref{thm-conv-lowerbound}]
	~\newline \indent
			Let  $0< r < \min\{ \operatorname{foc}(B_{\operatorname{inj}(p)}(p)), \frac{1}{2}\operatorname{inj}(p)\}$. 
			By Proposition \ref{prop-conv-focal-cutdecay}, it suffices to show that $R_c(p)>r$, that is, $B_r(p)$ contains no cut point of $q$ for any $q\in B_r(p)$.
			Let us argue by contradiction.
			If there is a cut point of $q$ in $B_r(p)$, say $x$, then the value of perimeter function $\left. P_{p;q}\right|_{\operatorname{Cut}(q)}:\operatorname{Cut}(q)\to \mathbb R$ at $x$ satisfies 
			$$P_{p;q}(x)=d(p,x)+d(x,q)+d(p,q)< 4r.$$
			Let $x_0$ be a minimum point of $\left. P_{p;q}\right|_{\operatorname{Cut}(q)}$, then 
			$$P_{p;q}(x_0)\le P_{p;q}(x)< 4r.$$
			Because for any two points $y,z$ in the triangle $\triangle(pqx_0)$, one always has
			$$d(y,z)\le \frac{1}{2}P_{p;q}(x_0),$$ it follows that
			$$d(x_0,q)<2r, \text{ and } d(x_0,p)<2r.$$
			If $q$ and $x_0$ are not conjugate to each other along a minimal geodesic between them, then by Lemma \ref{lem-klingenberg-generalized} there are two minimal geodesics from $p$ to $x_0$ and from $x_0$ to $q$ which form a whole geodesic $\gamma$.
			Now it is clear that for any $z$ on $\gamma$,
			$$d(p,z)\le \frac{1}{2} P_{p;q}(x_0)<2r,$$
			which implies that $\gamma$ lies in $B(p,2r)$.
			Therefore there are two geodesics connecting $p$ and $q$ inside $B_{2r}(p)$ with $2r<\operatorname{inj}(p)$, a contradiction. 
			
			To complete the proof, let us now verify that $q,x_0$ are not conjugate to each other along any minimal geodesic. This is because if they are conjugate along a minimal geodesic, then
			$$\operatorname{foc}(q)+\operatorname{foc}(x_0)\le d(q,x_0)<2r,$$
			which implies
			$$\min\{\operatorname{foc}(q),\operatorname{foc}(x_0)\})< r.$$
			Since $q,x_0\in B_{2r}(p)\subset B_{\operatorname{inj}(p)}(p)$,
			$$r<\operatorname{foc}(B_{\operatorname{inj}(p)}(p))\le \min\{\operatorname{foc}(q),\operatorname{foc}(x_0)\}<r,$$ 
			a contradiction derived. 
	\end{proof}
	
	\begin{proof}[Proof of Theorem 1.2]
		~\newline\indent
		For any $q\in B_{\operatorname{inj}(p)}(p)$, we want to show that $$\operatorname{inj}(q)\ge \min\{\operatorname{inj}(p),\operatorname{conj}(q) \}-d(p,q).$$
		Let us argue by contradiction.
		For $0<r<\min\{\operatorname{inj}(p),\operatorname{conj}(q) \}-d(p,q)$, assume that $B_r(q)$ contains a cut point $x$ of $q$, then the minimum of perimeter function $P_{p;q}$ on $\operatorname{Cut}(q)$ satisfies
		$$\min P_{p;q}|_{\operatorname{Cut}(q)}\le d(p,x)+d(q,x)+d(p,q)\le 2r+2d(p,q).$$
		Let $x_0$ be a minimal point of $\left.P_{p;q}\right|_{\operatorname{Cut}(q)}$, then 
		$$d(x_0,q)\le \frac{1}{2}P_{p;q}(x_0)\le r+d(p,q)<\operatorname{conj}(q).$$
		By Lemma \ref{lem-klingenberg-generalized}, 
		there are minimal geodesics $\beta:[0,d(p,x_0)]\to M$ from $p$ to $x_0$ and $\alpha:[0,d(q,x_0)]\to M$ from $x_0$ to $q$ such that they form a whole geodesic $\sigma=\alpha*\beta$ passing $x_0$.
		Because for any point $y$ on $\sigma$, 
		$$d(y,p)\le \frac{1}{2}P_{p;q}(x_0)\le r+d(p,q)<\operatorname{inj}(p),$$ 
		in the open ball $B_{\operatorname{inj}(p)}(p)$ there is another geodesic $\sigma$ connecting $p$ and $q$ other than the minimal geodesic between them, a contradiction.
	\end{proof}

	\begin{proof}[Proof of Theorem \ref{thm-geodesic-length}]
		$ $ \newline \indent
		Let $0<r\le\min\{\operatorname{inj}(p),\frac{1}{2}\operatorname{conj}(B_{\operatorname{inj}(p)}(p))\}$ and $\gamma:[0,1]\to B_r(p)$ is a geodesic. Because the image of $\gamma$ is compact, there is $0<r_1<r$ such that $\gamma$ lies in $B_{r_1}(p)$.
		Let us consider $T_pM$ with pull-back metric $\exp_p^*g$ by $\exp_p$ from $M$. Then $\exp_p^*g$ is non-degenerate on $B_{\operatorname{conj}(p)}(o)$. 
		Let $g_0=g|_p$ be the Euclidean metric on $T_pM$, then we define a complete metric $\tilde g$ by gluing  $\exp_p^*g$ and $g_0$ together such that $\left.\tilde g\right|_{B_{2r_1}(o)}=\exp_p^*g$, $\left.\tilde g\right|_{T_pM\setminus B_{\operatorname{conj}(p)}(o)}=g_0$. Then under the new metric $\tilde g$, $\operatorname{inj}(o)\ge 2r_1$.
		Let $\tilde\gamma:[0,1]\to B_r(o)$ be the lift of $\gamma$ in $B_{r_1}(o)\subset (T_pM,\tilde g)$. We claim that $\tilde\gamma$ is a minimal geodesic in $B_{\operatorname{conj}(p)}(o)$. 
		
		Indeed, if $\tilde\gamma(t)$ is a cut point of $v=\tilde\gamma(0)$, then by the proof of Theorem \ref{thm-conv-lowerbound}, the minimum point $w$ of $\left.P_{o;v}(x)\right|_{\operatorname{Cut}(v)}$ and the minimal geodesic connecting $v$ and $w$ lie in $B_{2r_1}(o)$. Because $2r_1<\operatorname{conj}(\exp_pv)$, $w$ is not a conjugate point of $v$. By the generalized Klingenberg's lemma \ref{lem-klingenberg-generalized}, there is another geodesic from $o$ to $v$ passing through $w$, which is a contradiction.
		
		Since $\tilde \gamma$ is a minimal geodesic, the length of $\tilde{\gamma}$ $< 2r_1$. Hence so is $\gamma$.
	\end{proof}

	Motivated by Lemma \ref{lem-klingenberg-improved} and \ref{lem-klingenberg-generalized}, we call a cut point $x$ of $q\in M$ is \emph{totally conjugate} \cite{Xu2016On} if $q$ and $x$ are conjugate to each other along every minimal geodesic between them, then we call $x$ a \emph{totally conjugate cut} point of $q$. 
	Now it is clear that
	Klingenberg's equality (\ref{eq-inj-pt}) on injectivity radius can be rewritten as
	\begin{equation}
	\operatorname{inj}(p)=\min\{\frac{1}{2}l(p), \operatorname{conj_{t}}(p)\},
	\end{equation}
	where the \emph{totally conjugate cut} radius, $\operatorname{conj_t}(p),$ is defined by
	$$\operatorname{conj_t}(q)=\sup\{r>0 : \text{$B_r(q)$ contains no totally conjugate cut point of $q$}\}.$$
	By the proofs of Theorem 1.2, it is also clear that (\ref{ineq-aroundinj}) can be improved to be
	\begin{equation}
	\operatorname{inj}(q)\ge \min\{\operatorname{inj}(p),\operatorname{conj_t}(q)\}-d(p,q).
	\end{equation}
	
	For the concentric convex radius of $p$, we have the following estimate in terms of the extended focal radius of $p$ and conjugate radius around $p$. 
	
	\begin{theorem}\label{thm-convpt-exfoc}
		Let $U=B_{\min\{\operatorname{foc_e}(p),\frac{1}{2}\operatorname{inj}(p)\}}(p)$.
		\begin{equation}\label{ineq-convpt-around}
		\begin{aligned}
		\operatorname{conv_{ct}}(p)
		&\ge\min\{\operatorname{foc_e}(p), \operatorname{\frac{1}{2}\operatorname{inj}(p)}, \frac{1}{2}\operatorname{conj_t}(U)\}\\
		&= \min\{\operatorname{foc_e}(p), \frac{1}{4}l(p), \frac{1}{2}\operatorname{conj_t}(U)\}
		\end{aligned}
		\end{equation}
	\end{theorem}
	\begin{proof}
		It follows from similar arguments as the proof of Theorem \ref{thm-conv-lowerbound}.
	\end{proof}
	We point it out that in (\ref{ineq-convpt-around}) the conjugate radius of points of $p$ is necessary in general. The vertex $p$ of a paraboloid of revolution satisfies that $\operatorname{conv_{ct}}(p)
	<\min\{\operatorname{foc_e}(p), \operatorname{\frac{1}{2}\operatorname{inj}(p)}\}=\infty$.

	\section{Discontinuity of Concentrically Convex Neighborhoods}
	In this section we give examples to illustrate that the concentric convexity radius $\operatorname{conv_{ct}}(p)$ may be not continuous and does not equal $\operatorname{conv}(p)$, as well as other examples mentioned earlier. 
	
	\begin{theorem}\label{thm-counterexample-euclidean}
		There are smooth rotationally symmetric metrics on $\mathbb{R}^n$ $(n\ge2)$, $h=dr^2+\varphi^2(r)dS_{n-1}^2$, without conjugate points such that the sectional curvature of $h$ changes signs and there is a sequence $z_i$ of points converging to $z$ with $\operatorname{conv_{ct}}(z_i)<2$ and $\operatorname{conv_{ct}}(z)=\infty$. 
	\end{theorem}
	By standard surgery arguments, Theorem \ref{thm-counterexample} follows from Theorem \ref{thm-counterexample-euclidean}. The metric in Theorem 4.1 can be chosen as one of that 
	Gulliver constructed in \cite{Gulliver1975On} whose sectional curvature changes signs with focal points and without conjugate points. Let $\frac{\pi}{4}<r_1<0.8$ and $r_2,\epsilon>0$ such that $\sin r_1=\sinh (r_1-r_2)$, $r_1-\epsilon>\frac{\pi}{4}$, $r_1+\epsilon\le 0.85$. According to \cite{Gulliver1975On}, a warped product metric $h=dr^2+\varphi^2(r)dS_{n-1}^2$ can be constructed such that
	\begin{enumerate}[(\ref{thm-counterexample-euclidean}.1)]
		\item  
		$\varphi(r)|_{[0,r_1-\epsilon]}=\sin(r)$, 
		$\varphi(r)|_{[r_1+\epsilon,+\infty]}=\sinh(r-r_2)$, and the sectional curvature $\operatorname{sec}_h(r)\le G(r)$, where $G$ is monotone non-increasing;
		\item $h$ has no conjugate points;
		\item \label{item-foc-rad-halfpi} there are points in $B_{r_1-\epsilon}(o)$ whose focal radius $=\frac{\pi}{2}$;
		\item \label{item-foc-rad} for any $z\in B_{r_1+\epsilon}(o)$, either  $\operatorname{foc}(z)\le 2$ or $\operatorname{foc}(z)=\infty$. And $\operatorname{foc}(o)=\infty$.
	\end{enumerate}  
	The points of focal radius $=\frac{\pi}{2}$ in (\ref{thm-counterexample-euclidean}.\ref{item-foc-rad-halfpi}) are endpoints of geodesic arc of length $\frac{\pi}{2}$ lying in the spherical cap $B_{r_1-\epsilon}(o)$. Since (\ref{thm-counterexample-euclidean}.\ref{item-foc-rad}) are not explicitly stated in
	\cite{Gulliver1975On}, for reader's convenient we give a sketched proof.
	
	\begin{proof}[Proof of (\ref{thm-counterexample-euclidean}.\ref{item-foc-rad})]
		We first show that for any unit-speed geodesic $\gamma$ starting at $x\in B_{r_1+\epsilon}(o)$ and any normal Jacobi field along $\gamma$ with $J(0)=0$, $|J|'$ is always positive outside $B_{2}(o)$. Because $\gamma$ is minimizing, it follows that $\gamma(t)$ leaves $B_{r_1+\epsilon}(o)$ after time $\le t_1=2(r_1+\epsilon)$ and never comes back. 
		Let $$f=\begin{cases}
		1, & 0\le t\le t_1,\\
		-1, & \text{otherwise}.
		\end{cases}$$
		and $u$ be a distributional solution of $u''+fu=0$ with
		$$u(0)=0 \text{ and } u'(0)=1.$$  
		By Lemma 3 in \cite{Gulliver1975On}, 
		$$\frac{|J|'}{|J|}\ge \frac{u'}{u} \quad\text{ on $[0,t^*]$,\qquad whenever $u>0$ on $(0,t^*]$}.$$ Since $u$ can explicitly solved, it can be directly checked that $u>0$ for all $t\ge 0$ and $u'>0$ whenever $t\ge 2$. Indeed, $u|_{[0,t_1]}=\sin t$, and
		$$u|_{[t_1,+\infty)}=
		\sin t_1 \cosh(t-t_1)+\cos t_1 \sinh(t-t_1).$$
		Because $t_1\le 1.7$, $$u|_{[t_1,+\infty)}>\sin t_1 \left(\cosh t-0.13\sinh t\right)>0,$$ and $u'|_{[2,+\infty)}>0$. Therefore $|J|'>0$ whenever $t\ge 2$.
		
		Similarly it can be showed that $\operatorname{foc}(o)=\infty$.
	\end{proof}
	
	We now prove Theorem \ref{thm-counterexample-euclidean}.
	
	\begin{proof}[Proof of Theorem \ref{thm-counterexample-euclidean}]
	$ $\newline\indent
	Let $h$ be defined as above. 
	Let $q$ be a point in $B_{r_1-\epsilon}(o)$ such that $\operatorname{foc}(q)=\frac{\pi}{2}$. Let $\gamma:[0,1]\to M$ be the minimal $h$-geodesic from $q$ to $o$. 
	Let $t_0=\max\{t\in[0,1]:\operatorname{foc}(\gamma(t))<\infty\}$. 
	By (\ref{thm-counterexample-euclidean}.\ref{item-foc-rad}), if $\operatorname{foc}(\gamma(t))<\infty$, then $\operatorname{foc}(\gamma(t))<2$.
	By the lower semi-continuity of $\operatorname{foc}(x)$, $t_0$ exists  and $\operatorname{foc}(\gamma(t_0))\le 2$. By the choice of $t_0$, it is clear that $\operatorname{foc_e}(\gamma(t_0))=\infty$.
	
	By Lemma \ref{lem-converg-foc-radi}, there is a sequence of point $z_i$ converging to $z=\gamma(t_0)$ such that $\lim_{i\to \infty}\operatorname{foc_e}(z_i)=\operatorname{foc}(z)<2.$ By (\ref{prop-conv-focal-cutdecay}.1), it is clear that $\operatorname{conv_{ct}}(z_i)=\operatorname{foc_e}(z_i)$ and $\operatorname{conv_{ct}}(z)=\infty$.  
	\end{proof}
	\begin{remark}
		By the symmetry of $h$ it is easy to see that $z_i$ can be chosen as points in $\left.\gamma\right|_{(0,t_0)}$ in the above proof of Theorem \ref{thm-counterexample-euclidean}.
	\end{remark}
	
	 It follows from Theorem \ref{thm-counterexample-euclidean} that there is a point $z$ such that $\operatorname{conv_{ct}}(z)>\operatorname{conv}(z)$. So in general the two concepts of convexity radius are different. 
	 
	 It is natural to ask what kind of manifolds satisfy that for any point $p$, the equality in (\ref{ineq-conv-foc}) always holds. By the equality (\ref{eq-conv-global}), if $\operatorname{foc}(B_{\operatorname{inj}(p)}(p))=\operatorname{foc}(M)$ and $\operatorname{inj}(p)=\operatorname{inj}(M)$, then it is clear that (\ref{ineq-conv-foc}) holds as an equality. In particular, we have
	\begin{example}\label{ex-homog}
		Let $M$ be a homogenous space. Then for any $p\in M$,
		$$\operatorname{conv}(p)=\min\{\operatorname{foc}(B_{\operatorname{inj}(p)}(p)), \frac{1}{2}\operatorname{inj}(p)\}.$$
	\end{example}  
	It can be easily seen that (\ref{ineq-conv-foc}) may be strict on a locally symmetric space. For example, $\operatorname{conv}(p)>\frac{1}{2}\operatorname{inj}(p)$ in $\mathbb{R}\times_{e^r}T^2$. In the end of this paper we give examples to show that (\ref{ineq-aroundinj}) is generally sharp.
	\begin{example}\label{ex-lip-sharp}
		For any $\epsilon>0$, let $(M,h)$ be a $2$-dimensional flat cone of angle $(1-\delta)\pi$, i.e.,
		$h=dr^2+r^2\frac{(1-\delta)^2}{4}d\theta^2$. Then any point $x\in M$ other than the vertex has injectivity radius $r(x)\cos \frac{\delta\pi}{2}$. Then any two points $p,q$ on a ray from the vertex satisfies that
		$$\left|\operatorname{inj}(p)-\operatorname{inj}(q)\right|=\cos \frac{\delta\pi}{2}\left|r(p)-r(q)\right|>(1-\epsilon)d(p,q),$$ for sufficient small $\delta>0$.
		Because $(M,h)$ can be made into a smooth manifold by gluing with a thin cylinder around the vertex without changing the injectivity radius of $p$ and $q$ above, this shows that the Lipschitz constant in the decay of (\ref{ineq-aroundinj}) is sharp in general.  
	\end{example}

	\bibliography{document}{}
	\bibliographystyle{plain}
	
\end{document}